\documentclass[a4paper,10pt]{article}
\usepackage[utf8]{inputenc}
\usepackage{amsthm}
\usepackage{amsmath}
\usepackage{bm}
\usepackage{amsfonts}
\usepackage{amssymb}
\usepackage{mathtools}
\usepackage{mathrsfs}
\usepackage{setspace}
\usepackage{xcolor}
\usepackage[a4paper, left=3.5cm, right=3.5cm, top=4cm, bottom=4cm]{geometry}

\usepackage{hyperref} 

\newcommand{\C}{\mathbb{C}}
\newcommand{\D}{\mathbb{D}}

\newcommand{\EE}{\mathbb{E}}
\newcommand{\F}{\mathcal{F}}

\newcommand{\M}{\mathcal{M}}
\newcommand{\N}{\mathbb{N}}

\newcommand{\PP}{\mathbb{P}}

\newcommand{\R}{\mathbb{R}}

\newcommand{\T}{\mathbb{T}}

\newcommand{\Z}{\mathbb{Z}}

\let\div\relax
\DeclareMathOperator{\div}{div}

\renewcommand{\epsilon}{\varepsilon}

\newcommand{\one}{\bm{1}}

\DeclareMathOperator{\imm}{i}

\newcommand{\de}{\partial}

\newcommand\restr[2]{{
  \left.\kern-\nulldelimiterspace
  #1
  \vphantom{\big|}
  \right|_{#2}
  }}

\newcommand{\set}[1]{\left\{#1\right\}}
\newcommand{\pa}[1]{\left(#1\right)}
\newcommand{\bra}[1]{\left[#1\right]}
\newcommand{\abs}[1]{\left|#1\right|}
\newcommand{\norm}[1]{\left\|#1\right\|}
\newcommand{\brak}[1]{\left\langle#1\right\rangle}

\newcommand{\expt}[1]{\mathbb{E}\left[#1\right]}

\makeindex

\title{Stationary Solutions of Damped Stochastic\\ 2-dimensional Euler's Equation}
\author{Francesco Grotto\\ Scuola Normale Superiore
	 \thanks{email address: francesco.grotto@sns.it}}
\date{20 January 2019}

\begin{document}

\newcommand{\secautorefname}{Section}

\newtheorem{thm}{Theorem}
\newcommand{\thmautorefname}{Theorem}
\newtheorem{defi}{Definition}
\newcommand{\defiautorefname}{Definition}
\newtheorem{cor}{Corollary}
\newcommand{\corautorefname}{Corollary}
\newtheorem{lem}{Lemma}
\newcommand{\lemautorefname}{Lemma}
\newtheorem{prop}{Proposition}
\newcommand{\propautorefname}{Proposition}

\theoremstyle{remark}
\newtheorem{rmk}{Remark}
\newcommand{\rmkautorefname}{Remark}
\newtheorem{ex}{Example}

\maketitle

\begin{abstract}
Existence of stationary point vortices solution to the damped and stochastically driven Euler's equation
on the two dimensional torus is proved, by taking limits of solutions with finitely many vortices. 
A central limit scaling is used to show in a similar manner the existence of stationary solutions
with white noise marginals.
\end{abstract}

\section{Introduction}

The present work concerns a particular class of solutions to the 2-dimensional incompressible Euler's equation 
with frictional damping, on the torus $\T^2= \R^2/\Z^2$,
\begin{equation}\label{modelvelocity}
\partial_t u_t+ u_t\cdot \nabla u_t+\nabla p_t=-\theta u_t + F_t, \quad \nabla\cdot u_t=0,
\end{equation}
where $u_t$ is the velocity vector field, $p_t$ is the (scalar) pressure, $\theta>0$ 
and $F_t$ is a stochastic forcing term. Our motivation stems from works on 2-dimensional turbulence:
our model can be regarded as an inviscid version of the one considered in \cite{boffettaecke}, which aimed to describe
the energy cascades phenomena in stationary, energy-dissipated, 2-dimensional turbulence.
Inspired by recent renewed theoretical interest for point vortices methods in the study of 2-dimensional Euler's equation
stemming from \cite{flandoli}, we will study solutions to (\ref{modelvelocity}) obtained as systems of 
interacting point vortices, and Gaussian limits of the latter ones. 
Even if our models are not able to capture turbulence phenomena such as the celebrated
energy spectrum decay law of inverse cascade predicted by Kolmogorov, we believe that the mechanism
of creation and damping of point vortices we describe might contribute to provide a description
of experimental behaviours of models such as the ones in \cite{boffettaecke}. Moreover,
the mathematical treatment of measure- or distribution-valued solution to Euler's equation
is not a trivial task, due to the need of quite weak notions of solution in presence of a singular nonlinearity.

From the mathematical viewpoint, equation (\ref{modelvelocity}) has been widely investigated especially as inviscid
limit of driven and damped Navier-Stokes equation, see for instance \cite{bessaihferrario}, \cite{constantin} and references
therein. Aside from the fact that we are dealing directly with the inviscid case, a substantial difference
of this work with respect to those ones is the space regularity of solutions. Indeed, existence and uniqueness
for 2-dimensional Euler equations are well established facts in spaces of suitably regular function spaces,
while the interesting case of solutions taking values in signed measures or distributions
remains quite open, especially in the uniqueness part:
we refer to \cite{majdabertozzi} for a general overview of the theory. The results of \cite{flandoli},
which we review in \autoref{ssec:weaksolutions},
established an important link between the theory of point vortices models and Gaussian invariant measures
to Euler's equation. We refer to \cite{majdabertozzi,marchioro} and to \cite{albeveriocruzeiro} for reviews on,
respectively, the former and latter ones. We also mention that limits of Gibbsian point vortices ensembles
(originally proposed by Onsager, \cite{onsager}) converging to Gaussian invariant measures were already considered for
instance in \cite{bpp} (the similarities between the two being already pointed out by Kraichnan \cite{kraichnan}).
However, Flandoli \cite{flandoli} was the first, as far as we know, to prove convergence of the system evolving
in time, as opposed to the simple convergence of invariant measures of the other ones.
His approach was based on the weak vorticity formulation of \cite{schochetweakvorticity} (see also its references),
which had already been considered in the point vortices model, \cite{schochetpointvortices}, 
and turned out to be suitable to treat solutions with white noise marginals.
Our results generalise the ones of \cite{flandoli} by
combining a stochastic forcing term (already considered in the vortices setting in \cite{flandoligubinellipriola},
or in function spaces in \cite{flandolimaurelli}) and damping.
Stationary solutions are regarded with particular interest in the theory, and the invariant distributions
we consider are also invariants of Euler's equation with no damping of forcing (see \autoref{rmk:nonlinearterm}):
to our knowledge, the Poissonian invariant distributions with infinite vortices we introduce below are new,
while their Gaussian counterpart (the \emph{enstrophy measure}, more generally known as white noise) 
have been an object of interest since the works of Hopf, \cite{hopf}.
For a more general discussion on invariant measures of Lévy type we refer to \cite{albeverioferrariolevy},
in which most of the basic ideas we rely upon are finely presented, although their arguments then
proceed along point of view of Dirichlet forms theory.

We will treat our model equation in vorticity form,
\begin{equation}\label{modelvorticity}
\partial_t \omega_t=-\theta \omega_t +u_t\cdot \nabla \omega_t+ \Pi_t, \quad \omega_t=\nabla^\perp\cdot u_t,
\end{equation} 
where $\nabla^\perp=(\partial_2,-\partial_1)$. The idea is to exhibit solutions by adapting the point
vortices model for Euler's equation, which, in absence of forcing and damping, we recall to be the measure valued solutions 
\begin{equation}\label{eulervortices}
\omega_t=\sum_{i=1}^N \xi_i \delta_{x_{i,t}}, 
\quad \dot x_{i,t}=\sum_{j\neq i} \xi_j \nabla^\perp \Delta^{-1}(x_{i,t},x_{j,t}),
\end{equation}
where $x_i\in\T^2$ is the position and $\xi_i\in\R$ the intensity of a vortex, to Euler's equation
\begin{equation}\label{eulereq}
\begin{cases}
\de_t \omega_t + u_t\cdot \nabla \omega_t =0\\
\nabla^\perp \cdot u_t=\omega_t,
\end{cases}
\end{equation}
(see \autoref{sec:definitions} for the appropriate notion of solution).
Inclusion of the damping term in our model will amount to an exponential quenching of the vortex intensities, with rate $\theta$.
Because of dissipation due to friction (which physically results from the 3-dimensional environment in which the 
2-dimensional flow is embedded), a forcing term is necessary in order for the model to exhibit stationary behaviour.
We will choose as $\Pi_t$ a Poisson point process, so to add new vortices and rekindle the system.
The linear part of (\ref{modelvorticity}), which is a Poissonian Ornstein-Uhlenbeck equation,
suggests that stationary distributions are made of countable vortices with
exponentially decreasing intensity, but in fact dealing with solutions of (\ref{modelvorticity}) having such marginals
seems to be as hard as the white noise marginals case.
The latter will be also addressed, taking as in \cite{flandoli} a ``central limit'' scaling of the vortices model,
resulting in solutions of (\ref{modelvorticity}) with space white noise marginal, and space-time white noise as forcing term. 

Our main result will be the \emph{existence} of solutions to (\ref{modelvorticity}) in these two cases:
infinite vortices marginals and Poisson point process forcing; white noise marginals and space-time white noise forcing.
The latter one draws us closer to the models in \cite{boffettaecke}, where the forcing term
was Gaussian with delta time-correlations. We will apply a compactness method: our approximant processes will not be
approximated solutions (as in Faedo-Galerkin methods), but true point vortices solutions with finitely many
vortices, for which we are able to prove well-posedness thanks to the techniques of \cite{marchioro}.

We regard the following results as a first step in the analysis of equation (\ref{modelvorticity}) by point vortices methods,
the natural prosecution being the study of driving noises with more complicated space correlations,
such as the ones used in numerical simulations reviewed in \cite{boffettaecke}.

\section{Preliminaries and Main Result}\label{sec:definitions}

Consider the the 2-dimensional torus $\T^2= \R^2/\Z^2$; we denote by
$H^\alpha=H^\alpha(\T^2)=W^{\alpha,2}(\T^2)$, for $\alpha\in\R$ the $L^2=L^2(\T^2)$-based Sobolev spaces, which enjoy the compact embeddings $H^\alpha\hookrightarrow H^\beta$ whenever $\beta<\alpha$,
the injections being furthermore Hilbert-Schmitd if $\alpha>\beta+1$.
Sobolev spaces are conveniently represented in terms of Fourier series:
let $e_k(x)=e^{2\pi ik\cdot x}$, $x\in\T^2$, $k\in\Z^2$, be the usual Fourier orthonormal basis: then
\begin{equation}\label{fourierseries}
H^\alpha=\set{u(x)=\sum_{k\in\Z^2}\hat u_k e_k(x): \norm{u}^2_{H^\alpha}= \sum_{k\in\Z^2} (1+|k|^2)^\alpha|\hat u_k|^2<\infty},
\end{equation}
where $\hat u_k=\bar{\hat u}_{-k}\in\C$ (we only consider real spaces).
We denote by $\M=\M(\T^2)$ the space of finite signed measures on $\T^2$: recall that measures have Sobolev regularity
$\M(D)\subset H^{-1-\delta}$ for all $\delta>0$ (for instance, because by dominated convergence 
their Fourier coefficients converge to constants). 
The brackets $\brak{\cdot,\cdot}$ will stand for $L^2$ duality couplings, such as the one between measures and continuous functions,
or between Sobolev spaces of opposite orders, unless we specify otherwise.
The capital letter $C$ will denote (possibly different) constants, and subscripts will point out 
occasional dependences of $C$ on other parameters.
Lastly, we write $X\sim Y$ when the random variables $X,Y$ have the same law.

\subsection{Random Variables}\label{ssec:randomvariables}

In order to lighten notation, in this paragraph we denote random variables (or stochastic processes)
and their laws with the same symbols. Let us also fix $H:=H^{-1-\delta}$, with $\delta>0$, the Sobolev space in which
we embed our random measures and distributions.
We will deal with stochastic objects of Gaussian and Poissonian nature: the former are likely to be the more familiar ones,
so we begin our review with them. We refer to \cite{peszatzabczyk, sato} for a complete discussion of the 
underlying classical theory.

Let $W_t$ be the cylindrical Wiener process on $L^2(\T^2)$, that is $\brak{W_t,f}$ is a real-valued centred Gaussian process
indexed by $t\in[0,\infty)$ and $f\in L^2(\T^2)$ with covariance
\begin{equation}\label{whitenoisecovariance}
\expt{\brak{W_t,f},\brak{W_s,g}}=t\wedge s \brak{f,g}_{L^2(\T^2)}
\end{equation}
for any $t,s\in[0,\infty)$ and $f,g\in L^2(\T^2)$. Since the embedding $L^2(\T^2)\hookrightarrow H^{-1-\delta}(\T^2)$ is
Hilbert-Schmidt, $W_t$ defines a $H^{-1-\delta}$-valued Wiener process. 
The law $\eta$ of $W_1$ is called the \emph{white noise} on $\T^2$, and it can thus be regarded as a Gaussian probability measure
on $H^{-1-\delta}$. Analogously, the law $\zeta$ of the (distributional) time derivative of $W$ can be identified both with a centred
Gaussian process indexed by $L^2([0,\infty)\times\T^2)$ and identity covariance operator or with a centred Gaussian probability
measure on $H^{-3/2-\delta}([0,\infty)\times\T^2)$; $\zeta$ is called the \emph{space-time white noise} on $\T^2$.

The couplings of $\eta$ against $L^2$ functions are called Ito-Wiener integrals: we will see that \emph{double}
Ito-Wiener integrals play a crucial role in this context, so let us recall their definition (for which we refer to \cite{janson}).
The double stochastic integral with respect to $\eta$ is the isometry $I^2:L_{sym}^2(\T^{2\times2})\rightarrow L^2(\eta)$ 
(which is not onto, the image being the second Wiener chaos)
defined extending by density the following expression on symmetric products:
\begin{equation*}
I^2(f\odot g)= :\brak{\eta,f}\brak{\eta,g}:=\brak{\eta,f}\brak{\eta,g}-\brak{f,g} \qquad \forall f,g\in L^2(D),
\end{equation*}
with $f\odot g(x,y)=\frac{f(x)g(y)+f(y)g(x)}{2}$. Equivalently, it is the extension by density of the map
\begin{equation}\label{doubleintegralsimplefunctions}
\sum_{\substack{i_1,i_2=1,\dots,n\\i_1\neq i_2}}a_{i_1,i_2}\one_{A_{i_1}\times A_{i_2}}\mapsto 
\sum_{\substack{i_1,i_2=1,\dots,n\\i_1\neq i_2}}a_{i_1,i_2} \eta(A_{i_1}) \eta(A_{i_2})
\end{equation}
where $n\geq 0$, $A_1,\dots,A_n\subset \T^2$ are disjoint Borel sets and $a_{i,j}\in\R$.
Let us compare it with another notion of double integral: 
considering $\eta$ as a random distribution in $H^{-1-\delta}$, the tensor product $\eta\otimes\eta$ is defined
as a distribution in $H^{-2-2\delta}(\T^{2\times2})$, so for $h\in H^{2+\delta}(\T^{2\times2})$ we can couple $\brak{h,\eta\otimes\eta}$.
For any $h\in H_{sym}^{2+\delta}(\T^{2\times2})$, it holds (as an equality between $L^2(\eta)$ variables)
\begin{equation}\label{doubleintegraldiagonalcontribution}
\brak{h,\eta\otimes\eta}=I^2(h)+\int_D h(x,x)dx
\end{equation}
(since it is true for the dense subset of symmetric products)
where we remark that $\int_D h(x,x)dx$ makes sense since $h$ has a continuous version.
We thus see that Ito-Wiener integration corresponds to ``subtract the diagonal contribution'' to the tensor product:
in order to make the dependence of double Ito-Wiener integrals on $\eta$, and motivated by the above discussion,
we will use in the following the notation
\begin{equation*}
:\brak{h,\eta\otimes\eta}:=I^2(h).
\end{equation*}

Besides those Gaussian distributions, we will be interested in a number of Poissonian variables, which we now define
in the framework of \cite{peszatzabczyk}. For $\lambda>0$, let $\pi^\lambda$ be the Poisson random measure on
$[0,\infty)\times H^{-1-\delta}$ with intensity measure $\nu$ given by the product of the measure $\lambda dt$ on $[0,\infty)$
and the image of $\sigma \delta_x$ where $\sigma=\pm1$ and $x\in\T^2$ are chosen uniformly at random.
In other terms, one can define the compound Poisson process on $H^{-1-\delta}$ (in fact on $\M$),
\begin{equation}\label{defsigma}
\Sigma_t^\lambda=\sum_{i:t_i\leq t}\sigma_i \delta_{x_i}=\int_0^t d\pi^\lambda,
\end{equation}
starting from the jump times $t_i$ of a Poisson process of parameter $\lambda$, a sequence $\sigma_i$ of i.i.d. $\pm1$-valued 
Bernoulli variable of parameter $1/2$ and a sequence $x_i$ of i.i.d uniform variables on $\T^2$.
Notice that, since its intensity measure has $0$ mean, $\pi^\lambda$ is a compensated Poisson measure,
or equivalently $\Sigma_t^\lambda$ is a $H^{-1-\delta}$-valued martingale. 
Moreover, $\Sigma_t^\lambda$ has the same covariance of the cylindrical Wiener process $W_t$ (up to the factor $\lambda$):
\begin{equation}\label{poissoncovariance}
\expt{\brak{\Sigma_t^\lambda,f}\brak{\Sigma_s^\lambda,g}}=\lambda (t\wedge s) \brak{f,g}_{L^2}^2,
\end{equation}
and also the same quadratic variation,
\begin{equation}\label{quadraticvariation}
\bra{\brak{\Sigma^\lambda,f}}_t=\lambda t \norm{f}_{L^2}^2.
\end{equation}

We will need a symbol for another Poissonian integral, the $H^{-1-\delta}$-valued (in fact $\M$-valued) variable
\begin{equation}\label{defxi}
\Xi^{\lambda,\theta}_M =\sum_{i:t_i\leq M} \sigma_i e^{-\theta t_i} \delta_{x_i}=\int_0^M e^{-\theta t} d\pi^\lambda,
\end{equation}
where $M,\theta>0$. Thanks to the negative exponential, the above integrals converge also when $M=\infty$, 
defining a random measure: we will call it $\Xi^{\lambda,\theta}=\Xi^{\lambda,\theta}_\infty$.

\begin{rmk}\label{rmk:vorticesmeasures}
	By (\ref{defxi}), a sample of the random measure $\Xi^{\lambda,\theta}_M$ is a finite sum of \emph{point vortices}
	$\xi_i \delta_{x_i}$ with $\xi_i\in\R, x_i\in\T^2$.
	We will say that the random vector $(\xi_i,x_i)_{i=1\dots N}\in(\R\times\T^2)^N$
	(with random length $N$) is sampled under $\Xi^{\lambda,\theta}_M$ if $\sum_{i=1}^N \xi_i \delta_{x_i}$ has the law of
	$\Xi^{\lambda,\theta}_M$. Analogously (and in a sense more generally speaking), 
	the sequence $(t_i,\sigma_i,x_i)_{i\in\N}$ is sampled under $\pi^\lambda$ if the sum of $\sigma_i \delta_{t_i}\delta_{x_i}$
	has the law of the Poisson point process $\pi^\lambda$.
\end{rmk}

Our Poissonian measures are characterised by their Laplace transforms: for any measurable and bounded
$f:\T^2\rightarrow \R$,
\begin{align}
\expt{\exp\pa{\alpha\brak{f,\Sigma_t^\lambda}}}&=
\exp\pa{\lambda t\int_{\set{\pm1}\times \T^2} (e^{\alpha\sigma f(x)}-1)d\sigma dx},\\
\label{laplacexi}
\expt{\exp\pa{\alpha\brak{f,\Xi^{\lambda,\theta}_M}}}&=
\exp\pa{\lambda \int_{[0,M]\times\set{\pm1}\times \T^2} (e^{\alpha\sigma e^{-\theta t} f(x)}-1)dtd\sigma dx},
\end{align}
where $d\sigma$ denotes the uniform measure on $\pm1$. By the isometry property of Poissonian integrals,
the second moments of $\Sigma_t^\lambda$ and $\Xi^{\lambda,\theta}_M$ are given by
\begin{equation*}
\expt{\norm{\Sigma_t^\lambda}^2_{H^{-1-\delta}}}=C \lambda t,\quad
\expt{\norm{\Xi^{\lambda,\theta}_M}^2_{H^{-1-\delta}}}= C \frac{\lambda}{\theta}(1-e^{-\theta M}),
\end{equation*}
where $C=\norm{\delta}^2_{H^{-1-\delta}}$ is the Sobolev norm of Dirac's delta.

In this Poissonian case, we can define double integrals against functions $h\in H_{sym}^{2+\delta}(\T^{2\times2})$
(which are continuous) $\PP$-almost surely as
\begin{align}
\brak{h,\Xi^{\lambda,\theta}_M\otimes\Xi^{\lambda,\theta}_M}
&=\sum_{i,j:t_i,t_j\leq M} \sigma_i\sigma_j e^{-\theta(t_i+t_j)} h(x_i,x_j),\\
\label{doubleintegralpoisson}
:\brak{h,\Xi^{\lambda,\theta}_M\otimes\Xi^{\lambda,\theta}_M}:
&=\sum_{\substack{i,j:t_i,t_j\leq M\\i\neq j}} \sigma_i\sigma_j e^{-\theta(t_i+t_j)} h(x_i,x_j),\\
\end{align} 
where the second one is defined in analogy with (\ref{doubleintegralsimplefunctions}),
and, as in that case, in fact it extends to an isometry $L_{sym}^2(\T^{2\times2})\rightarrow L^2(\PP)$,
\begin{equation*}
\expt{\abs{:\brak{h,\Xi^{\lambda,\theta}_M\otimes\Xi^{\lambda,\theta}_M}:}^2}=
\frac{\lambda^2}{\theta}(1-e^{-\theta M})^2 \norm{h}_{L^2(\T^{2\times2})}^2.
\end{equation*}
Let us note here that (\ref{doubleintegralpoisson}) also extends to functions $h$ that are
smooth outside the diagonal set ${(x,x): x\in\T^2}\subset \T^{2\times 2}$,
but possibly discontinuous or singular on it: this is going to be important in the sequel.

An important link between the objects we just defined is the following: 
\begin{prop}[Ornstein-Uhlenbeck process]\label{prop:oup}
	Consider the $H^{-1-\delta}$-valued linear stochastic differential equation
	\begin{equation}\label{ouequation}
	d u_t=-\theta u_t dt+ d\Pi_t.
	\end{equation}
	If $\Pi_t=\sqrt{\lambda}W_t$, there exists a unique stationary solution
	with invariant measure $\sqrt{\frac{\lambda}{2\theta}}\eta$, and if $u_0\sim C \eta$ ($C>0$), 
	the invariant measure is approached exponentially fast, $u_t\sim \sqrt{\frac{\lambda}{2\theta}(1-e^{-2\theta t}(1-C^2))}\eta$.
	
	Analogously, if $\Pi_t=\Sigma^\lambda_t$, there exists a unique stationary solution 
	with invariant measure $\Xi^{\theta,\lambda}_\infty$, and if $u_0\sim\Xi^{\theta,\lambda}_M$,
	then $u_t$ will have law $\Xi^{\theta,\lambda}_{M+t}$ for any later time $t>0$.
\end{prop}
The linear equation (\ref{ouequation}), in both the outlined cases, has a unique $H^{-1-\delta}$-valued strong solution,
with continuous trajectories in the Gaussian case, and \emph{cadlag} trajectories in the Poissonian one.
Well-posedness of the linear equation and uniqueness of the invariant measure are part of the classical
theory (see \cite{peszatzabczyk}), and they descend from the explicit solution by stochastic convolution:
\begin{equation}\label{ouexplicit}
u_t=e^{-\theta t}u_0+\int_0^t e^{-\theta (t-s)}d\Pi_s,
\end{equation}
from which it is not difficult to derive also the last statement of the Proposition.

\subsection{Weak solutions of 2D Euler equation}\label{ssec:weaksolutions}

We now review some definitions of measure-valued and distribution-valued solutions to the 2D Euler's equation:
the point is how to make sense of the multiplication appearing in the nonlinearity.
The equation in terms of the vorticity $\omega$ is (\ref{eulereq})
\begin{equation*}
\begin{cases}
\de_t \omega_t + u_t\cdot \nabla \omega_t =0\\
\nabla^\perp \cdot u_t=\omega_t,
\end{cases}
\end{equation*}
and it has to be complemented with boundary conditions: on the torus $\T^2$ one should impose that $\omega_t$ have zero average.
However, since we are dealing with a conservation law, the space average is not involved in the dynamics (it is constant).
\begin{rmk}
	We will henceforth deliberately ignore the zero average condition:
	it will always be possible to subtract a constant number (constant in time and space, but possibly a random variable) to
	take care of it, but we refrain from doing so to avoid a superfluous notational burden.
\end{rmk}

Let $G$ be the Green function of $\Delta$ on $\T^2$ with zero average, and let $K=\nabla^\perp G$ be
the Biot-Savart kernel; the former has the explicit representation
\begin{equation*}
G(x,y)=-\frac{1}{2\pi}\sum_{k\in \Z^2} \log(d_{\R^2}(x,y+k)).
\end{equation*}
We will use the fact that $|\nabla G(x,y)|,|K(x,y)|\leq \frac{C}{d(x,y)}$ for all $x,y$, with $C$ a universal constant.
The second equation of (\ref{eulereq}) can be inverted by means of the Biot-Savart kernel:
we can write $u_t=K\ast \omega_t$, and thus obtain an equation where only $\omega$ appears.
Its integral form against a smooth test function $f$ is
\begin{equation}\label{eulerintegral}
\brak{f,\omega_t}=\brak{f,\omega_0}+\int_0^t \int_{\T^{2\times 2}} K(x,y)\cdot \nabla f(x) \omega_s(x)\omega_s(y)dxdyds
\end{equation}
(keeping in mind that $\nabla\cdot\nabla^\perp \omega\equiv 0$ to perform integration by parts),
which can be symmetrised (swapping $x$ and $y$) into
\begin{equation}\label{eulerdelort}
\brak{f,\omega_t}=\brak{f,\omega_0}+\int_0^t \int_{\T^{2\times 2}} H_f(x,y)\omega_s(x)\omega_s(y)dxdyds
\end{equation}
where $H_f(x,y)=\frac{1}{2}K(x,y)(\nabla f(x)-\nabla f(y))$ is a bounded symmetric function, smooth outside the diagonal set
\begin{equation*}
\triangle^2:=\set{(x,x)\in \T^{2\times 2}}.
\end{equation*}
These three formulations are equivalent for smooth $\omega_t$,
but the integral forms, especially the symmetrised one, have been used to define more general solutions
of Euler's equation, see \cite{schochetweakvorticity}. 
One such solution is the system of (finitely many) \emph{Euler's point vortices}:
the evolution of the vorticity $\omega_t=\sum_{i=1}^N \xi_i \delta_{x_{i,t}}$ (with $\xi_i\in\R$ and $x_i\in \T^2$)
is given by (\ref{eulervortices}),
\begin{equation*}
\dot x_{i,t}=\sum_{j\neq i} \xi_j K(x_{i,t},x_{j,t}). 
\end{equation*}
This model is thoroughly discussed for instance in \cite{marchioro}, where it is remarked that it satisfies (\ref{eulerintegral})
if the double space integral is taken outside the diagonal $\triangle^2$, where $K$ is singular:
\begin{equation*}
\int_{\T^{2\times 2}\setminus\triangle^2} K(x,y)\cdot \nabla f(x) \omega_s(x)\omega_s(y)dxdy=\sum_{i\neq j}\xi_i \xi_j K(x_i,x_j)\cdot \nabla f(x_i).
\end{equation*}
It is thus possible, in sight of the notation we introduced in (\ref{doubleintegralpoisson}),
to formulate Euler equation in the point vortices case as follows:
if $\omega_t=\sum_{i=1}^N \xi_i \delta_{x_{i,t}}$ and we denote
\begin{equation*}
:\brak{H_f,\omega_t\otimes\omega_t}: =\sum_{i\neq j}^N\xi_i \xi_j H_f(x_i,x_j),
\end{equation*}
then it holds
\begin{equation}\label{normalisedeulerequation}
\brak{f,\omega_t}=\brak{f,\omega_0}+\int_0^t :\brak{H_f,\omega_t\otimes\omega_t}:ds.
\end{equation}
The need to avoid the diagonal set $\triangle^2$ in order to give meaning to singular solutions is going to be crucial in what follows,
as it is in the proof of the forthcoming important well-posedness result.
\begin{prop}[Marchioro-Pulvirenti]\label{prop:marchioro}
	Let $\xi_1,\dots,\xi_n\in\R$ and $x_1,\dots,x_n\in \T^2$.
	For almost every initial data $x_{1,0},\dots,x_{n,0}\in \T^2$ under the $n$-fold product of Lebesgue's measure, 
	the system of differential equations (\ref{eulervortices})
	has a smooth, global in time solution $x_{1,t},\dots,x_{n,t}$, which preserves the product measure on the initial condition.
	The measure-valued process $\omega_t=\sum_{i=1}^n \xi_i \delta_{x_i}$ then satisfies (\ref{eulerintegral}) in the sense above.
\end{prop}
(In fact the latter is a slight generalisation of the results in \cite{marchioro}, which will be a consequence of the further
generalisation we will prove in \autoref{sec:finitevortices}.)

In \cite{flandoli}, Flandoli performed a scaling limit of the point vortices system to exhibit
(stationary) solutions with space white noise marginals:
the meaning of the equation for such irregular vorticity processes was understood by carrying to the limit the formulation (\ref{normalisedeulerequation}), since, as we have seen in the last paragraph, the Wiener-Ito interpretation
of the nonlinear term makes perfect sense in the case of white noise. To proceed rigorously, let us give the following:

\begin{defi}\label{defi:solution}
	Let $(\omega_t)_{t\in[0,T]}$ be a $H^{-1-\delta}$-valued continuous stochastic process defined on a probability space $(\Omega,\F,\PP)$,
	with fixed time marginals $\omega_t$ having the law of white noise $\eta$ for all $t\in[0,T]$.
	We say that $\omega$ is a weak solution to Euler's equation if for any $f\in C^\infty(\T^2)$, $\PP$-almost surely, 
	for any $t\in[0,T]$,
	\begin{equation}\label{eulerdouble}
	\brak{f,\omega_t}=\brak{f,\omega_0}+\int_0^t :\brak{H_f,\omega_s\otimes\omega_s}: ds.
	\end{equation}
\end{defi}

\begin{rmk}
	Notice that the Ito-Wiener integrals (in space) appearing in the definition are almost surely integrable in time
	since their $L^2(\PP)$ norms are uniformly bounded in $t$. The latter definition coincides with the one of \cite{flandoli},
	only, in that article, it was not observed that the approximation procedure used to define the nonlinear term
	in fact coincides with the classic Ito-Wiener integral.
\end{rmk}

The formulation (\ref{eulerdouble}) is in fact quite general: interpreting the colons as 
``subtraction of the diagonal contribution'', this formulation might include all deterministic solutions, 
both in the classical and weak formulation (\ref{eulerdelort}) (\emph{cf.} \cite{schochetweakvorticity}),
the point vortices solution of \autoref{prop:marchioro}, and it is the sense in which
the limit process with white noise marginals of \cite{flandoli} solves Euler's equation.

\begin{prop}[Flandoli]\label{prop:flandoli}
	There exists a stationary stochastic process $\omega_t$ with fixed-time marginals $\omega_t\sim \eta$ 
	and trajectories of class $C([0,T],H^{-1-\delta})$ for any $\delta>0$  which is a solution of Euler's equation in the sense of \autoref{defi:solution}. 
\end{prop}

\begin{rmk}
	In fact, \cite{flandoli} proves the same result also for processes with fixed-time marginals $\omega_t\sim \rho_t(\eta)$
	for suitable functions $\rho:[0,T]\times H^s(\T^2)\rightarrow \R$.
\end{rmk}

For the sake of completeness, we recall that solutions to Euler's equation with white noise marginals were first
built in \cite{albeverioribeiro}, by means of Galerkin approximation on $\T^2$. 

\subsection{Main Results}\label{sec:mainresult}

Fix $\lambda,\theta>0$. Our model is the stochastic differential equation
\begin{equation}\label{model}
d \omega=-\theta \omega dt+(K\ast \omega)\cdot \nabla \omega dt+ d\Pi_t,
\end{equation}
where $d\Pi_t$ is either the Poisson process $d\Sigma^\lambda_t$ or the space-time white noise $dW_t$.
We have seen in \autoref{prop:oup} how the linear part of the equation behaves;
the intuition provided by the point vortices system suggests that, thanks to the Hamiltonian form
of the nonlinearity, the latter only contributes to ``shuffle'' the vorticity without changes to the fixed time statistics.
This intuition can be motivated as follows. Since the point vortices system preserves the product Lebesgue measure,
the system must preserve the Poissonian random measures $\Xi_{M}^{\lambda,\theta}$ we introduced in
\autoref{ssec:randomvariables}, because the positions of vortices under those measures are uniformly, independently scattered
(this fact will be rigorously proved in \autoref{sec:finitevortices} for $M<\infty$).
Building Gaussian solutions by approximation with Poissonian ones thus must produce the same phenomenon.
In other words, with an eye towards stationary solutions, 
we expect to be able to build a Poissonian stationary solution with $\omega_t\sim \Xi^{\theta,\lambda}_\infty$ in the case
$\Pi_t=\Sigma^\lambda_t$, and a stationary Gaussian solution with $\omega_t\sim \sqrt{\frac{\lambda}{2\theta}}\eta$
in the case $\Pi_t=\sqrt{\lambda} W_t$. 

\begin{rmk}
	These claims are deeply related with the fact that 2D Euler's equation preserves \emph{enstrophy}, $\int_{\T^2}\omega(x)^2dx$,
	when smooth solutions are considered. The quadratic form associated to enstrophy, that is the $L^2(\T^2)$ product,
	is (up to multiplicative constants) the covariance of random fields $\Xi_{M}^{\lambda,\theta}$ and $\eta$:
	as already remarked in \cite{albeverioferrariolevy}, one should expect all random fields with such covariance to be invariant for Euler's equation, even if the very meaning of the latter sentence has to be clarified.
\end{rmk}

First and foremost, we need to specify a suitable concept of solution:
inspired by the discussion of the last paragraph, we give the following one.
\begin{defi}\label{defi:modelsolution}
	Fix $T,\delta>0$, and let $(\Omega,\F,\PP, \F_t)$ be a probability space with a filtration $\F_t$ satisfying the usual hypothesis,
	with respect to which  $(\Pi_t)_{t\in[0,T]}$ is a $H^{-1-\delta}$-valued $\F_t$-martingale. 
	Let $(\omega_t)_{t\in[0,T]}$ be a $H^{-1-\delta}$-valued $\F_t$-predictable process, with trajectories of class
	\begin{equation}\label{regularityassumption}
	L^2([0,T],H^{-1-\delta})\cap \D([0,T],H^{-3-\delta})
	\end{equation}
	for some $q\geq 1$ ($\D([0,T],S)$ denotes the space of $S$-valued \emph{cadlag} functions into a metric space $S$).
	We consider the cases:
	\begin{itemize}
		\item[(P)]  $\Pi_t=\Sigma^\lambda_t$ and $\omega_t\sim \Xi^{\lambda,\theta}_{M+t}$ for all $t\in[0,T]$,
		with $0\leq M<\infty$;
		\item[(Ps)]  $\Pi_t=\Sigma^\lambda_t$ and $\omega_t\sim \Xi^{\lambda,\theta}_{\infty}$ for all $t\in[0,T]$;
		\item[(G)] $\Pi_t=\sqrt{\lambda}W_t$ and $\omega_t\sim \sqrt{\frac{\lambda}{2\theta}(1-e^{-2\theta(M+t)})}\eta$
		for all $t\in[0,T]$, with $0\leq M<\infty$;
		\item[(Gs)] $\Pi_t=\sqrt{\lambda}W_t$ and $\omega_t\sim \sqrt{\frac{\lambda}{2\theta}}\eta$
		for all $t\in[0,T]$.
	\end{itemize}
	We say that $(\Omega,\F,\PP,\F_t,\Pi_t,\omega_0,(\omega_t)_{t\in[0,T]})$ is a \emph{weak solution} of (\ref{model}) if
	for any $f\in C^\infty(\T^2)$ it holds $\PP$-almost surely for any $t\in[0,T]$: 
	\begin{equation}\label{weakmodel}
	\brak{f,\omega_t}=e^{-\theta t}\brak{f,\omega_0}+\int_0^t e^{-\theta(t-s)} :\brak{H_f,\omega_s\otimes\omega_s}: ds +\int_0^t e^{-\theta(t-s)} \brak{f,d\Pi_s}.
	\end{equation}
	If instead, given $(\Omega,\F,\PP,\F_t,W_t)$ there exists a process $\omega_t$
	as above, we call it a \emph{strong solution}.
\end{defi}

\begin{rmk}\label{rmk:integralformulation}
	Equation (\ref{weakmodel}) is motivated in sight of (\ref{ouexplicit}) and (\ref{eulerdouble}).
	The ``variation of constants'' expression in the above definition is equivalent to the ``integral'' one
	\begin{equation}\label{weakmodelsimple}
	\brak{f,\omega_t}=\brak{f,\omega_0}-\theta \int_0^t \brak{f,\omega_s} ds
	+\int_0^t :\brak{H_f,\omega_s\otimes\omega_s}: ds +\brak{f,\Pi_t},
	\end{equation}
	as one can verify integrating by parts in time. Both versions will be useful in what follows, but we deem (\ref{weakmodel})
	the most suggestive.
\end{rmk}

\begin{rmk}\label{rmk:nonlinearterm}
	The nonlinear term of (\ref{weakmodel}) is well-defined thanks to the isometry properties of Gaussian
	and Poissonian double integral (see \autoref{sec:definitions}): indeed, the integrand is bounded in $L^2(\PP)$ uniformly in time,
	so that, in particular, $\int_0^t :\brak{H_f,\omega_s\otimes\omega_s}: ds$ is a continuous function of time.
\end{rmk}

We are now able to state our main result.

\begin{thm}\label{thm:mainresult}
	There exist weak solutions of (\ref{model}) in all the outlined cases, stationary 
	(as $H^{-1-\delta}$-valued stochastic processes) in the cases (Ps) and (Gs). 
\end{thm}

As already remarked, equation (\ref{model}) is difficult to deal with directly in the Gaussian 
(or even the stationary Poisson) case:
for instance it does not seem possible to treat it with fixed point or semigroup techniques.
We prove existence of stationary solutions by taking limits of point vortices solutions, 
corresponding to the case (P).
We begin with a solution $\omega_{M}$ of the equation (\ref{model}) with noise $\Sigma^\lambda_t$ starting 
from finitely many vortices distributed as $\Xi_M^{\theta,\lambda}$. Well-posedness in this case is ensured
by a generalisation of \autoref{prop:marchioro}, whose proof is the content of \autoref{sec:finitevortices}.
The first limit we consider is $M\rightarrow\infty$, so to build a stationary solution with invariant measure 
$\Xi^{\theta,\lambda}$ and thus obtain existence in case (Ps).
Scaling intensities $\sigma\rightarrow \frac{\sigma}{\sqrt N}$ 
and generation rate $\lambda\rightarrow N\lambda$, we prove that as $N\rightarrow \infty$ 
the limit points are stationary solutions of (\ref{model}) driven by space-time white noise and with invariant measure the space white noise. The nonstationary Gaussian case (G) will be derived analogously, in this sort of central limit theorem.

We are applying a \emph{compactness method}: first, we prove probabilistic bounds on the involved distribution, in order to
-second step- apply a compactness criterion ensuring tightness of the approximating processes; finally, we pass to the limit
the equation satisfied by the approximants. 

\begin{rmk}\label{rmk:invarianteuler}
	Consider the case when no damping or forcing are present: we noted above that the classical finite vortices system
	(\ref{eulervortices}) preserves the product Lebesgue's measure, so in particular the distributions
	$\Xi_{M}^{\theta,\lambda}$ with $M<\infty$ and $\theta,\lambda>0$ are also invariant.
	The very same limiting procedure we are going to use, as $M\rightarrow\infty$, proves existence of stationary solutions
	to Euler's equation in its weak formulation (\ref{eulerdouble})
	with invariant measure $\Xi_{\infty}^{\theta,\lambda}$ (or $\eta$, the case of \cite{flandoli}), 
	where the definition of solution is to be given in the fashion of \autoref{defi:solution}.
	More generally, Poissonian and Gaussian stationary solutions, as suggested in \cite{albeverioferrariolevy},
	should be particular cases of stationary solutions with independently scattered random distributions.
\end{rmk}

\section{Solutions with finitely many vortices}\label{sec:finitevortices}

Even in the case of initial data distributed as $\Xi_M^{\lambda,\theta}$, that is with almost surely finitely many
initial vortices, solving the nonlinear equation
\begin{equation}\label{modelpoisson}
d \omega=-\theta \omega dt+(K\ast \omega)\cdot \nabla \omega dt+ d\Sigma^\lambda_t
\end{equation}
is not a trivial task. We will build a solution describing explicitly how the initial vortices and the ones
added by the noise term evolve, as a system of increasingly numerous differential equations for the positions 
of vortices $x_i$. 
Intuitively, the process $\omega_{M,t}$ is defined as follows: from the initial datum $\omega_{M}(0)$, 
which is sampled under $\Xi_M^{\theta,\lambda}$, we let the system evolve according to the deterministic dynamics
\begin{equation*}
\dot x_i= \sum_{j\neq i}\xi_j e^{-\theta t} K(x_i,x_j)
\end{equation*}
until the first jump time $t_1$ of the driving noise $\Sigma^\lambda_t$,
when we add the vortex corresponding to the jump, and so on.
To treat the model rigorously, let us introduce the following notation:
let $x_{1,0},\dots,x_{n,0}$ and $\xi_{1,0},\dots,\xi_{n,0}$ be the (random) positions and signs of vortices of the initial datum,
and set for notational convenience $t_1=\dots=t_n=0$ their birth time;
at time $t_i$ it is added a vortex with intensity $\xi_{i,t_i}=\pm 1$ in the position $x_{i,t_i}$, but we can pretend it to actually have existed since time 0, and just come into play at the time $t_i$. Thus, our equations are
%
\begin{align}
\label{posdyn}
x_{i,t}&= x_{i,t_i}+\one_{t_i\leq t}\int_{t_i}^t \sum_{j\neq i:t_j\leq s} \xi_{j,s} K(x_{i,s},x_{j,s})ds,\\
\label{intdyn}
\xi_{i,t}&=\begin{cases}\xi_{i,0} &t<t_i \\e^{-\theta(t-t_i)}\xi_{i,0} & t\geq t_i. \end{cases}      
\end{align}
In this formulation of the problem, part of the randomness consists in the positions and intensities of the 
initial vortices and the ones to be: the random jump times $t_i$ then determine when the latter ones become part of the system.
Let us thus fix the $t_i$'s (that is, condition the process given the distribution of the $t_i$'s) so to reduce us to
a deterministic problem with random initial data. The existence of a solution for almost every initial condition is ensured
by the following generalisation of Proposition \ref{prop:marchioro}.
\begin{prop}\label{timedepmarchioro}
	Let $(x_{i,0})_{i\in\N}$ be a sequence of i.i.d uniform variables on $\T^2$.
	For every locally finite sequence of jump times $0\leq t_1\leq\dots\leq t_i\leq \dots\leq\infty$
	and initial intensities $(\xi_{i,0})\in[-1,1]$ the system of equations (\ref{posdyn}) and (\ref{intdyn})
	possesses a unique, piecewise smooth and \emph{cadlag}, global in time solution, for a full probability set which does not depend on the
	choice of $t_i,\xi_{i,0}$. 
	At any time, the joint law of positions $x_i$ is the infinite product of Lebesgue measure on $\T^2$.
\end{prop}

We use the hypothesis that the jump times $t_i$ are locally finite (there are only finitely many of them in every compact $[0,T]$)
so to reduce ourselves to a system of finitely many vortices.
In fact, we repeat the proof of \cite{marchioro} adapting it to our context.
The issue is the possibility of collapsing vortices, which is ruled out as follows.
We define an approximating system with interaction kernel smoothed in a ball around 0:
the smooth interaction readily gives well-posedness of the approximants,
on which we evaluate a Lyapunov functional measuring how close the vortices can get.
Bounding the Lyapunov function then ensures that as the regularisation parameter
goes to 0, the approximant vortices in fact perform the same motion prescribed by
the non-smoothed equation.

\begin{proof}
	Let $\delta>0$, and consider smooth functions $G_\delta$ coinciding with $G$ outside the fattened diagonal 
	$\set{(x,y)\in D^2:d(x,y)<\delta}$ ($d$ being the distance on the torus $\T^2$), and such that
	\begin{equation}\label{defgdelta}
	|G_\delta(x,y)|\leq C |G(x,y)|, \quad |\nabla G_\delta(x,y)|\leq \frac{C}{d(x,y)} \quad \forall x,y\in D.
	\end{equation}
	Note in particular that the latter inequality was already true for $G$. Let us first restrict ourselves to
	a time interval $[0,T]$: in particular, we can consider only the finitely many vortices with $t_i\leq T$,
	let them be $x_1,\dots,x_n$.
	The system with smoothed interaction kernel $K_\delta=\nabla^\perp G_\delta$ has a unique, 
	global in time, smooth solution thanks to Cauchy-Lipschitz theorem:
	let $x_{i,t}^\delta$ denote the solution (note that smoothing $K$ does not effect the evolution
	of the intensities $\xi_{i,t}$). 
	
	Because of the Hamiltonian structure of the equations, that is, since $K_\delta=\nabla^\perp G_\delta$,
	it holds $\div \dot x^\delta_{i,t}=0$. This implies the invariance of product Lebesgue measure:
	for any $f\in C^1(D^n)$,
	\begin{align*}
	\frac{d}{dt} &\int_{D^n} f(x^\delta_{1,t},\dots,x^\delta_{n,t})dx_{1,0}\dots dx_{n,0}\\
	&=\sum_{i=1}^n \int_{D^n} \nabla_i f(x^\delta_{1,t},\dots,x^\delta_{n,t}) 
	\dot x^\delta_{i,t}  dx_{1,0}\dots dx_{n,0}\\
	&=-\int_{D^n} f(x^\delta_{1,t},\dots,x^\delta_{n,t}) \div(\dot x^\delta_{i,t}) dx_{1,0}\dots dx_{n,0}=0.
	\end{align*}
	
	Let us now introduce a Lyapunov function measuring how close the existing vortices are by means of $G_\delta$:
	\begin{equation*}
	L_\delta(t)=L_\delta(t,x^\delta_{1,t},\dots,x^\delta_{n,t})=-\sum_{i\neq j: t_i,t_j\leq t} G_\delta(x^\delta_{i,t},x^\delta_{j,t}).
	\end{equation*}
	By replacing $G_\delta$ with $G_\delta-k$ for a large enough $k>0$ in the definition of $L_\delta$ we can assume that
	$L_\delta$ is nonnegative.
	Observe that, because of (\ref{defgdelta}), $\int_{D^n} L_\delta(0)dx_1,\dots dx_n\leq C$ for a constant $C$ independent of $\delta$.
	Upon differentiating, and keeping in mind that
	\begin{equation*}
	\dot x^\delta_{i,t}=\one_{t_i\leq t}\sum_{j\neq i:t_j< t} \xi_{j,t}\nabla^\perp G_\delta(x^\delta_{i,t},x^\delta_{j,t}),
	\end{equation*}
	we get
	\begin{align*}
	\frac{d}{dt} L_\delta(t) 
	&=-\sum_{i\neq j: t_i,t_j\leq t} \nabla G_\delta(x^\delta_{i,t},x^\delta_{j,t})
	\cdot(\dot x^\delta_{i,t}+\dot x^\delta_{j,t})\\
	&=\sum_{i,j,k\leq n} \tilde a_{ijk}(t) \nabla G_\delta(x^\delta_{i,t},x^\delta_{j,t})\cdot \nabla^\perp G_\delta(x^\delta_{i,t},x^\delta_{k,t}),
	\end{align*}
	where $\tilde a_{ijk}(t)$ depend on time $t$ as functions of the intensities $\xi_{i,t}$, $\tilde a_{ijk}=0$ whenever two indices are equal,
	since $\nabla G_\delta(x^\delta_{i,t}-x^\delta_{j,t})\cdot \nabla^\perp G_\delta(x^\delta_{i,t}-x^\delta_{j,t})=0$
	and it always holds $|\tilde a_{ijk}(t)|\leq 1$. We can use this to prove the following integral bound on $L_\delta$:
	denoting by $dx^n$ the $n$-fold Lebesgue measure of the distribution of initial position,
	\begin{align*}
	\int_{D^n} &\sup_{t\in[0,T]} L_\delta(t) dx^n\leq \int_{D^n} L_\delta(0) dx^n \\
	&\quad+\sum_{i,j,k} \int_0^T \int_{D^n} \abs{\nabla G_\delta(x^\delta_{i,s},x^\delta_{j,s})
		\cdot \nabla^\perp G_\delta(x^\delta_{i,s},x^\delta_{k,s})} dx^n ds\\
	&\leq \int_{D^n} L_\delta(0) dx^n+T C_n \int_{D^3} \abs{\nabla G_\delta(x,y)\cdot \nabla^\perp G_\delta(x,z)}dxdydz\\
	&\leq C_T,
	\end{align*}
	$C_T$ being a constant depending only on $T$ ($n$ depends on $T$). 
	Note that in the second inequality we have used the invariance of Lebesgue's measure.
	The last passage follows from the aforementioned integrability
	of $L_\delta(0)$ and the fact that, because of (\ref{defgdelta}), the integrands in the second term are bounded by
	\[\abs{\nabla G_\delta(x-y)\cdot \nabla^\perp G_\delta(x-z)}\leq \frac{C}{|x-y||x-z|}.\]
	With these estimates at hand, we can now pass to the limit as $\delta\rightarrow 0$: let
	\begin{equation*}
	d_{\delta,T}(x^n)=\min_{t\in[0,T]}\min_{i\neq j} d(x^\delta_{i,t}-x^\delta_{j,t}),
	\end{equation*}
	so that
	\begin{equation*}
	d_{\delta,T}(x^n)<\delta\Rightarrow \sup_{t\in[0,T]} L_\delta(t)>-C \log(\delta),
	\end{equation*}
	since when two points $x,y$ are closer than $\delta$, $G_\delta(x,y)\geq C \log(\delta)$ for some universal constant $C$.
	As a consequence, by \v{C}eby\v{s}\"ev's inequality,
	\begin{equation*}
	\PP(\Omega_{\delta,T}):=\PP(d_{\delta,T}(x^n)<\delta)\leq C'(-\log \delta)^{-1}.
	\end{equation*}
	By construction, in the event $\Omega^c_{\delta,T}$ the solution $x^\delta_{i,t}$ is in fact a solution of
	the original system in $[0,T]$. Hence, the thesis holds if the event
	\begin{equation*}
	\bar\Omega=\bigcup_{T>0} \bigcap_{\delta>0}\Omega_{\delta,T}
	\end{equation*}
	is negligible. But this is true: $\Omega_{\delta,T}$ is monotone in its arguments,
	so that the intersection in $\delta$ is negligible because of the above estimates,
	hence the increasing union in $T$ must be negligible too.
\end{proof}

The forthcoming Corollary is a direct consequence of Proposition \ref{timedepmarchioro}:
indeed to complete our construction we only need to randomise the jump times and intensities so that
the initial conditions and driving noise have the correct distribution.
Assume that
\begin{itemize}
	\item $(x_{1,0},\xi_{1,0}),\dots (x_{n,0},\xi_{n,0})$ are positions and intensities of vortices sampled under $\Xi_M^{\theta,\lambda}$,
	\item $(t_{n+m}, x_{n+m,0},\xi_{n+m,0}=\sigma_{n+m})_{m\geq 1}$ is sampled under $\pi^\lambda$,
\end{itemize}
both in the sense of \autoref{rmk:vorticesmeasures}, with variables defined on a probability space $(\Omega,\F,\PP)$.
Then there exists a piecewise smooth, \emph{cadlag} solution of the system of equations (\ref{posdyn}) and (\ref{intdyn}) for all $t\in[0,\infty)$, 
$\PP$-almost surely. Moreover, the positions of vortices at any time $t$, $x_{i,t}$, are i.i.d. uniform variables on the torus $\T^2$.

\begin{cor}\label{cor:finitevortices}
	In the outlined setting, the process $\omega_{M,t}=\sum_{i:t_i\leq t}\xi_{i,t} \delta_{x_{i,t}}$
	is a $\M$-valued \emph{cadlag} Markov process with fixed time marginals $\omega_{M,t}\sim\Xi^{\theta,\lambda}_{M+t}$
	for all $t\geq 0$. It is a strong solution of
	\begin{equation*}
	d \omega_M=-\theta \omega_M dt+(K\ast \omega_M)\cdot \nabla \omega_M dt+ d\Sigma^{\lambda}_t,
	\end{equation*}
	in the sense of \autoref{defi:modelsolution}
\end{cor}

\begin{proof}
	Fix $s<t$: by construction, given the positions $x_{i,0}$, the initial intensities $\xi_{i,0}$ and the jump times $t_i$
	(in a $\PP$-full measure event), $\omega_{M,t}$ is given by a deterministic function of $(x_{i,s},\xi_{i,s})_{i:t_i< s}$
	and $(t_i,x_{i,0},\xi_{i,0})_{i:s\leq t_i< t}$. As a consequence, $\omega_{M,t}$ is a function of $\omega_{M,s}$ and of
	the driving noise $(\Sigma^\lambda_r)_{s\leq r< t}$, which is independent from $\omega_{M,s}$: this implies
	the Markov property.
	Since the trajectories of positions $x_{i,t}$ and the evolution of intensities $\xi_{i,t}$ are smooth in time,
	$\omega_{M,t}$ is also smooth in time, save for the jump times $t_i$ when a new Dirac's delta is added.
	
	As for the marginal distributions, let us first evaluate:
	\begin{align*}
	\expt{e^{\imm\alpha \brak{\omega_{M,t},f}}} &= \expt{\exp\pa{\imm\alpha \sum_{i:t_i\leq t} \xi_{i,t} f(x_{i,t})}}\\
	&=\expt{\expt{\left. \exp\pa{\imm\alpha \sum_{i:t_i\leq t} \xi_{i,t} f(x_{i,t})}\right| (t_i)_{i\geq 0}}}\\
	&=\expt{\prod_{i:t_i\leq t} \int_D e^{\imm\alpha \xi_{i,t} f(x)} dx}=:\expt{\prod_{i:t_i\leq t} F(\xi_{i,t})}.
	\end{align*}
	Using the definition of $\xi_{i,t}$, and distinguishing the cases $i\leq n$ and $i>n$
	(which correspond to two independent groups of random variables), we can write
	\begin{align*}
	\expt{e^{\imm\alpha \brak{\omega_{M,t},f}}}&= \EE_N\bra{\prod_{s_i\in [0,M]}F(e^{-\theta s_i})}\cdot
	\EE_N\bra{\prod_{s_i\in [0,t]}F(e^{-\theta (t-s_i)})}\\
	&=\EE_N\bra{\prod_{s_i\in [0,M+t]}F(e^{-\theta s_i})}
	\end{align*}
	where $N$ is a Poisson point process of parameter $\lambda$ on $\R$ whose points are denoted by $s_i$,
	and the second passage follows from the fact that the points $N$ in disjoint intervals are independent and their
	distribution does not change if we reverse the parametrisation of the interval.
	Comparing to the characteristic function of $\Xi_{M+t}$ given in (\ref{laplacexi}), we conclude that $\omega_{M,t}\sim\Xi^{\theta,\lambda}_{M+t}$.
	
	Observe now that in this case it holds, for any $f\in C^\infty(\T^2)$, $\PP$-almost surely for all $t\geq 0$,
	\begin{equation*}
	:\brak{H_f,\omega_{M,t}\otimes\omega_{M,t}}:=\sum_{\substack{i,j:t_i,t_j\leq t\\i\neq j}} 
	\xi_{i,t}\xi_{j,t} H_f(x_{i,t},x_{j,t}),
	\end{equation*}
	(\emph{cf.} with \autoref{ssec:randomvariables}). Given this, it is straightforward to show that we do have built
	solutions of (\ref{weakmodel}): for $f\in C^\infty(\T^2)$, by (\ref{posdyn}) and (\ref{intdyn}),
	\begin{align*}
	\brak{f,\omega_{M,t}}&=\sum_{i:t_i\leq t} \xi_{i,t} f(x_{i,t})\\
	&=\sum_{i:t_i\leq t} \xi_{i,t} \pa{f(x_{i,t_i})+\int_{t_i}^t \sum_{j\neq i:t_j\leq s}\xi_{j,s} \nabla f(x_{i,s})\cdot  K(x_{i,s},x_{j,s})ds}\\ 	
	&=\pa{\sum_{i=1}^n+\sum_{i>n:t_i\leq t}}\xi_{i,t} f(x_{i,t_i})+\sum_{i:t_i\leq t} \xi_{i,t}
	\int_{t_i}^t \sum_{j\neq i:t_j\leq s}\xi_{j,s} \nabla f(x_{i,s})\cdot  K(x_{i,s},x_{j,s})ds\\
	&=\sum_{i=1}^n e^{-\theta t}f(x_{i,0})+\sum_{i>n:t_i\leq t} e^{-\theta(t-t_i)}f(x_{i,t_i})\\
	&\quad + \int_0^t \sum_{\substack{i,j:t_i,t_j\leq s\\i\neq j}} 
	e^{-\theta(t-s)}\xi_{i,s}\xi_{j,s}\nabla f(x_{i,s})\cdot  K(x_{i,s},x_{j,s})ds\\
	&=e^{-\theta t}	\brak{f,\omega_{M,0}}+\int_0^t e^{-\theta(t-s)} \brak{f,d\Sigma_s}
	+\int_0^t :\brak{H_f,\omega_{M,s}\otimes\omega_{M,s}}: ds.
	\end{align*}
	The latter equation holds regardless of the choice of initial positions, intensities and jump times (as soon as the 
	dynamics is defined) so in particular it holds $\PP$-almost surely uniformly in $t$, and this concludes the proof.
\end{proof}

The method of \cite{marchioro} thus provides, quite remarkably, existence and pathwise uniqueness of measure-valued strong solutions.
Unfortunately, it only seems to apply to systems of finitely many vortices, since it relies on the very particular, 
discrete nature of the measures involved to control the ``diagonal collapse'' issue. 
Let us conclude this section noting that we have obtained the first piece of \autoref{thm:mainresult},
namely we have built solutions in the case (P) for all $M<\infty$.

\section{Proof of the Main Result}

In \autoref{sec:finitevortices} we built the point vortices processes
$\omega_{M,t}=\sum_{i:t_i\leq t}\xi_{i,t} \delta_{x_{i,t}}.$
Let us introduce the scaling in $N\geq 1$: we will denote
$\omega_{M,N,t}=\sum_{i:t_i\leq t}\frac{\xi_{i,t}}{\sqrt N} \delta_{x_{i,t}}$
where $x_{i,t},\xi_{i,t}$ solve equations (\ref{posdyn}) and (\ref{intdyn}), and where the $t_i$'s
are the jump times of a real valued Poisson process of intensity $N\lambda$.
In other words, by \autoref{cor:finitevortices}, $\omega_{M,N,t}$ is a strong solution of
\begin{equation}\label{eqomegamn}
d \omega_{M,N}=-\theta \omega_{M,N} dt+(K\ast \omega_{M,N})\cdot \nabla \omega_{M,N} dt+ \frac{1}{\sqrt N}d\Sigma^{N\lambda}_t,
\end{equation}
(in the sense of \autoref{defi:modelsolution}) with fixed time marginals
$\omega_{M,N,t}\sim \frac{1}{\sqrt N} \Xi_{M+t}^{\theta,N\lambda}$.
It is worth to note here that, by construction of $\omega_{M,N,t}$, its natural filtration $\F_t$
coincides with the one generated by the driving noise $\Sigma^{N\lambda}_t$ and the initial datum.

The forthcoming paragraphs deal with, respectively: a recollection of some compactness criterions,
the bounds proving that the laws of $\omega_{M,N}$ are tight,
the proof of the fact that limit points of our family of processes are indeed solutions 
in the sense of \autoref{defi:modelsolution}, that is, the main result.

\subsection{Compactness Results}\label{ssec:simon}

Let us first review a deterministic compactness criterion due to Simon
(we refer to \cite{simon} for the result and the required generalities on Banach-valued Sobolev spaces).

\begin{prop}[Simon]
	Assume that
	\begin{itemize}
		\item $X\hookrightarrow B\hookrightarrow Y$ are Banach spaces such that the embedding $X\hookrightarrow Y$
		is compact and there exists $0<\theta<1$ such that for all $v\in X\cap Y$
		\begin{equation*}
		\norm{v}_B\leq M\norm{v}_X^{1-\theta} \norm{v}_Y^\theta;
		\end{equation*}
		\item $s_0,s_1\in\R$ are such that $s_\theta=(1-\theta)s_0+\theta s>0$.
	\end{itemize}
	If $\F \subset W$ is a bounded family in
	\begin{equation*}
	W=W^{s_0,r_0}([0,T],X)\cap W^{s_1,r_1}([0,T],Y)
	\end{equation*}
	with $r_0,r_1\in[0,\infty]$, and we define
	\begin{equation*}
	\frac{1}{r_\theta}=\frac{1-\theta}{r_0}+\frac{\theta}{r_1}, \qquad s_*=s_\theta-\frac{1}{r_\theta},
	\end{equation*}
	then if $s_*\leq 0$, $\F$ is relatively compact in $L^p([0,T],B)$ for all $p<-\frac{1}{s_*}$.
	In the case $s_*>0$, $\F$ is moreover relatively compact in $C([0,T],B)$.
\end{prop}

Let us specialise this result to our framework. Take
\begin{equation*}
X=H^{-1-\delta}(D), \quad B=H^{-1-\delta}(D), \quad Y=H^{-3-\delta}(D),
\end{equation*}
with $\delta>0$: by Gagliardo-Niremberg estimates the interpolation inequality is satisfied with $\theta=\delta/2$.
Let us take moreover $s_0=0$, $s_1=1/2-\gamma$ with $\gamma>0$, $r_1=2$ and $r_0=q\geq1$, so that
the discriminating parameter is
\begin{equation*}
s_*=-\gamma\theta-\frac{1-\theta}{q}.
\end{equation*}
Note that as we take $\delta$ smaller and smaller, and $q$ bigger and bigger, we can get $s_*<0$ arbitrarily close to $0$, but not $0$. We have thus derived:

\begin{cor}\label{cor:simon}
	If the sequence
	\begin{equation*}
	\set{v_n}\subset L^p([0,T],H^{-1-\delta})\cap W^{1/2-\gamma,2}([0,T],H^{-3-\delta})
	\end{equation*}
	is bounded for any choice of $\delta>0$ and $p\geq 1$, and for some $\gamma>0$, then it is relatively compact in
	$L^q([0,T],H^{-1-\delta})$ for any $1\leq q<\infty$.
	As a consequence, if a sequence of stochastic processes $u^n:[0,T]\rightarrow H^{-1-\delta}$ defined on a probability space $(\Omega,\F,\PP)$ is such that, for any $\delta>0$, $p\geq1$ and some $\gamma>0$,
	there exists a constant $C_{\delta,\gamma,q}$ for which
	\begin{equation}\label{momentscondition}
	\sup_n \expt{\norm{u^n(t)}_{ L^p([0,T],H^{-1-\delta})}+\norm{u^n}_{W^{1/2-\gamma,1}([0,T],H^{-3-\delta})}}
	\leq C_{\delta,\gamma,p},
	\end{equation} 
	then the laws of $u_n$ on $L^q([0,T],H^{-1-\delta})$ are tight for any $1\leq q<\infty$.
\end{cor}

The processes we will consider are discontinuous in time: this is why we consider only fractional Sobolev
regularity in time. However, as we have just observed, this prevents us to use Simon's criterion to
prove any time regularity beyond $L^q$. This is why we will combine the latter result with a compactness criterion for 
\emph{cadlag} functions. We refer to \cite{metivier} for both the forthcoming result and the necessary
preliminaries on the space $\D([0,T], S)$ of \emph{cadlag} functions
taking values in a complete separable metric space $S$.

\begin{thm}[Aldous' Criterion]\label{thm:aldous}
	Consider a sequence of stochastic processes $u^n:[0,T]\rightarrow S$ defined on probability spaces
	$(\Omega^n,\F^n,\PP^n)$ and adapted to filtrations $\F^n_t$. The laws of $u^n$ are tight on $\D([0,T],S)$ if:
	\begin{enumerate}
		\item for any $t\in[0,T]$ (a dense subset suffices) the laws of the variables $u^n_t$ are tight;
		\item for all $\epsilon,\epsilon'>0$ there exists $R>0$ such that for any sequence of $\F^n$-stopping times
		$\tau_n\leq T$ it holds
		\begin{equation*}
		\sup_n \sup_{0\leq r\leq R} \PP^n\pa{d(u^n_{\tau_n},u^n_{\tau_n+r})\geq \epsilon'}\leq \epsilon.
		\end{equation*} 
	\end{enumerate}
\end{thm}

\subsection{Tightness of Point Vortices Processes}\label{ssec:tightness}

The following estimate on our Poissonian random measures is the crux in all the forthcoming bounds;
it is essentially a Poissonian analogue of the ones in Section 3 of \cite{flandoli}. 

\begin{prop}
	Let $\omega_{M,N}\sim \frac{1}{\sqrt N} \Xi_{M}^{\theta,N\lambda}$. For any $1\leq p<\infty$ 
	there exists a constant $C_p>0$ such that for any measurable bounded functions $h:\T^2\rightarrow \R$
	and $f:\T^{2\times 2}\rightarrow \R$ it holds 
	\begin{equation}\label{doubleintegralestimate}
	\expt{\brak{h,\omega_{M,N}}^{2p}}\leq C_p \norm{h}_\infty^{2p},
	\qquad \expt{\brak{f,\omega_{M,N}\otimes \omega_{M,N}}^p}\leq C_p \norm{f}_\infty^p,
	\end{equation}
	uniformly in $N\geq 0$ and $M\in[0,\infty]$. As a consequence, since for $\delta>0$ the Green function 
	$\Delta^{-1-\delta}$ is smooth,
	\begin{equation}\label{uniformbound}
	\expt{\norm{\omega_{M,N}}_{H^{-1-\delta}}^{2p}}=\expt{\brak{\Delta^{-1-\delta},\omega_{M,N}\otimes\omega_{M,N}}^p}\leq C_{p,\delta},
	\end{equation}
	uniformly in $M,N$.
\end{prop}

\begin{proof}
	Since
	\begin{equation*}
	\brak{f,\omega_{M,N}\otimes \omega_{M,N}}=\brak{\tilde f,\omega_{M,N}\otimes \omega_{M,N}}, \quad 
	\tilde f(x,y)=\frac{1}{2}(f(x,y)+f(y,x)),
	\end{equation*}
	we reduce ourselves to symmetric functions. Moreover, without loss of generality we can check (\ref{doubleintegralestimate}) for functions with separate variables $f(x,y)=h(x)h(y)$, $h:\T^2\rightarrow \R$ measurable and bounded, for which it holds
	\begin{equation*}
	\expt{\brak{f,\omega_{M,N}\otimes \omega_{M,N}}^p}=\expt{\brak{h,\omega_{M,N}}^{2p}}.
	\end{equation*}
	Moments of the random variable $\brak{h,\omega_{M,N}}$ can be evaluated by differentiating the moment generating function
	(\ref{laplacexi}): using Fa\`a di Bruno's formula to take $2p$ derivatives we get
	\begin{align*}
	&\expt{\brak{h,\omega_{M,N}}^{2p}}=\\
	&\quad =(2p!)\sum_{\substack{r_1,\dots,r_{2p}\geq 0\\ r_1+2r_2+\dots +2pr_{2p}=2p}}
	\prod_{k=1}^{2p} \frac{1}{(k!)^{r_k}r_k!}\pa{N\lambda \int_{[0,M]\times\set{\pm1}\times \T^2} \frac{\sigma^k}{N^{k/2}}
		e^{-\theta t k}h(x) d\sigma dx dt}^{r_k}\\ 	
	&\quad \leq (2p!)\sum_{\substack{r_1,\dots,r_{2p}\geq 0\\ r_1+2r_2+\dots +2pr_{2p}=2p}}
	\prod_{k=1}^{2p} \frac{(N\lambda)^{r_k} \norm{h}_k^{k r_k} \one_{2|k}^{r_k}}{(\theta k)^{r_k} N^{kr_k/2}(k!)^{r_k}r_k!}\\
	&\quad =\frac{(2p!)\norm{h}_\infty^{2p}}{N^{p}}\sum_{\substack{r_1,\dots,r_{2p}\geq 0\\ r_1+2r_2+\dots +2pr_{2p}=2p}}
	\prod_{k=1}^{2p} \frac{(N\lambda)^{r_k} \one_{2|k}^{r_k}}{(\theta k)^{r_k}(k!)^{r_k}r_k!}
	\end{align*} 
	(see \cite{peszatzabczyk,privaultpoisson} for similar classical computations).
	Let us stress that when an integral in the latter formula is null, its 0-th power is to be interpreted as $0^0=1$.
	The contribution of $\one_{2|k}=\int \sigma^k d\sigma$ is crucial: 
	when $k$ is odd, $\one_{2|k}$ is null, so only terms with $m_k=0$ survive in the sum (again, $0^0=1$).
	Thus, the highest power of $N$ appearing is $N^{r_2}\leq N^{2p/2}=N^p$, which is compensated by the $N^{-p}$ we factored out,
	and this concludes the proof.
\end{proof}

We can now discuss convergence at fixed times.

\begin{prop}\label{prop:fixedtimemarginals}
	The laws of a family of variables $\omega_{M,N}\sim \frac{1}{\sqrt N} \Xi_{M}^{\theta,N\lambda}$, defined on a 
	probability space $(\Omega,\F,\PP)$ and taking values in on $H^{-1-\delta}$ are tight, for any fixed $\delta>0$.
	Moreover,
	\begin{itemize}
		\item the limit as $M\rightarrow\infty$ at fixed $N$, say $N=1$, is the law of $\Xi_{\infty}^{\theta,\lambda}$;	
		\item the limit as $N\rightarrow\infty$ at fixed $M$ (any $M\in (0,\infty]$) is the law of 
		$\sqrt{\frac{(1-e^{-2\theta M})\lambda}{2\theta}}\eta$;
	\end{itemize}
	and if the variables converge almost surely, they do so also in $L^p(\Omega,H^{-1-\delta})$ for any $1\leq p<\infty, \delta>0$.
\end{prop}

\begin{proof}
	The embedding $H^\alpha\hookrightarrow H^\beta$ is compact as soon as $\alpha>\beta$, and we know that the variables are
	uniformly bounded elements of $L^p(\Omega,H^{-1-\delta})$ for any $p\geq 1$ by (\ref{doubleintegralestimate}),
	so by \v{C}eby\v{s}\"ev's inequality their laws are tight.
	
	Identification of limit laws is yet another consequence of (\ref{laplacexi}): by Theorem 2 of \cite{gross}
	(an infinite-dimensional Lévy theorem)
	we only need to check that characteristic functions $\expt{e^{\imm \brak{\omega_{M,N},h}}}$ 
	converge to the ones of the announced limits for any $h\in H^{1+\delta}$.
	Since (\ref{laplacexi}) is valid for all $M\in[0,\infty]$, the limit for $M\rightarrow \infty$ poses no problem.
	As for the limit $N\rightarrow\infty$, for any test function $h\in H^{1+\delta}$,
	\begin{align*}
	\expt{\exp\pa{ \imm\brak{h,\omega_{M,N}}}} 
	&=e^{-N\lambda}\exp\pa{N\lambda  \int_{[0,M]\times\set{\pm1}\times \T^2}
		\exp\pa{\frac{\imm \sigma}{\sqrt N} h(x)e^{-\theta t}}dxd\sigma dt}\\
	&= e^{-N\lambda} \exp\pa{N\lambda \int_0^M \frac{1}{N}\norm{h}_2^2 e^{-2\theta t}dt+O_h\pa{\frac{1}{N}}}\\
	&\xrightarrow{N\rightarrow\infty} \exp\pa{\frac{\lambda }{2\theta} \norm{h}_2^2 (1-e^{-2\theta M})},
	\end{align*}
	where in the second step we used the following elementary expansion: for $\phi\in C(\T^2)$,
	\begin{equation}\label{elementaryexpansion}
	\abs{\frac{1}{2}\int_{\T^2} \pa{\exp\pa{\frac{\phi(x)}{\sqrt N}}+\exp\pa{-\frac{\phi(x)}{\sqrt N}}}dx-1-\frac{\norm{\phi}_2^2}{2N}}
	\leq \frac{\norm{\phi}_4^4}{24N^2}.
	\end{equation}
	Since $\expt{\exp\pa{ \imm\brak{h,\eta}}}=\exp(-\norm{h}_2^2)$, this concludes the proof.
\end{proof}

The latter result provides compactness ``in space'' (``equi-boundedness''): in order to apply
\autoref{cor:simon} and \autoref{thm:aldous}, we also need to obtain a control on the regularity ``in time'' 
(``equi-continuity''). We will obtain it by exploiting the equation satisfied by $\omega_{M,N}$, which we derived in
\autoref{cor:finitevortices}, which allows us to prove the forthcoming estimate on increments.

\begin{prop}
	Let $\omega_{M,N}:[0,T]\rightarrow H^{-1-\delta}$ be the stochastic process defined at the beginning of this Section.
	For any $\F_t$-stopping time $\tau\leq T$ (possibly constant), $r,\delta>0$,
	there exists a constant $C_{\delta,T}$ independent of $M,N,\tau,r$ such that
	\begin{equation}\label{incrementbound}
	\expt{\norm{\omega_{M,N,\tau+r}-\omega_{M,N,\tau}}_{H^{-3-\delta}}^2}\leq C_{\delta,T} \cdot r. 
	\end{equation} 
\end{prop}

\begin{proof}
	In order to lighten notation, and since the final result must not depend on $M,N$, let us drop them when writing
	$\omega_{M,N,t}=\omega_t$. By its definition in \ref{eqomegamn} and \autoref{rmk:integralformulation} we know that
	the process satisfies the integral equation
	\begin{equation}\label{increments}
	\brak{f,\omega_{t+r}}-\brak{f,\omega_t}=-\theta \int_t^{t+r} \brak{f,\omega_s} ds
	+\int_t^{t+r} :\brak{H_f,\omega_s\otimes\omega_s}: ds +\brak{f,\frac{1}{\sqrt{N}}(\Sigma^{N\lambda}_{t+r}-\Sigma^{N\lambda}_t)},
	\end{equation}
	for any smooth $f\in C^\infty(\T^2)$. Since this equation holds $\PP$-almost surely uniformly in $s,t\in[0,T]$,
	it is also true when we replace $t$ with the stopping time $\tau$.
	It is convenient to recall that
	\begin{equation*}
	\norm{u}_{H^{-3-\delta}}^2=\sum_{k\in\Z^2} (1+|k|^2)^{-3-\delta} |\hat u_k|^2,
	\end{equation*}
	so we can use the weak integral equation against the orthonormal functions $e_k$ to control the full norm:
	\begin{equation}\label{fullnorm}
	\expt{\norm{\omega_{\tau+r}-\omega_\tau}_{H^{-3-\delta}}^2}=\sum_{k\in\Z^2} (1+|k|^2)^{-3-\delta}
	\expt{\abs{\brak{\omega_{\tau+r}-\omega_\tau, e_k}}^2}.
	\end{equation}
	We estimate increments by bounding separately the terms in the equation, let us start from the linear one:
	\begin{equation}\label{pieceone}
	\expt{\abs{\int_\tau^{\tau+r} \brak{f,\omega_s} ds}^2} 
	\leq  r \expt{\int_0^T \abs{\brak{f,\omega_s} }^2 ds}=r \int_0^T \expt{\abs{\brak{f,\omega_s} }^2} ds
	\leq C T r \norm{f}_\infty^2, 		
	\end{equation}
	where the last passage makes use of the uniform estimate (\ref{doubleintegralestimate}).
	The nonlinearity is the harder one, and its singularity is the reason why we can not obtain space regularity
	beyond $H^{-3-\delta}$,
	\begin{align}\label{piecetwo}
	\expt{\abs{\int_\tau^{\tau+r} :\brak{H_f,\omega_s\otimes\omega_s}: ds}^2} 
	&\leq r \int_0^T  \expt{\abs{:\brak{H_f,\omega_s\otimes\omega_s}: }^2 }ds\\
	&\leq C T r \norm{H_f}_\infty^2\leq C T r \norm{f}_{C^2(\T^2)}^2,
	\end{align}
	where the second passage uses (\ref{doubleintegralestimate}), and the third is due to the fact that
	by Taylor expansion
	\begin{equation*}
	\abs{H_f(x,y)}=\frac{1}{2}\abs{K(x,y)(\nabla f(x)-\nabla f(y))}
	\leq C \frac{\abs{\nabla f(x)-\nabla f(y)}}{d(x,y)} \leq C \norm{f}_{C^2(\T^2)}.
	\end{equation*}
	By (\ref{quadraticvariation}), the martingale $(\brak{f,N^{-1/2}(\Sigma^{N\lambda}_{t+r}-\Sigma^{N\lambda}_t)})_{t\in[0,T]}$
	has constant quadratic variation $\lambda r \norm{f}^2_{L^2}$, so Burkholder-Davis-Gundy inequality gives
	\begin{align}\label{piecethree}
	\expt{\abs{\brak{f,N^{-1/2}(\Sigma^{N\lambda}_{\tau+r}-\Sigma^{N\lambda}_\tau)}}^2}
	\leq \expt{\sup_{t\in [0,T]} \abs{\brak{f,N^{-1/2}(\Sigma^{N\lambda}_{t+r}-\Sigma^{N\lambda}_t)}}^2}
	\leq C \lambda r \norm{f}^2_{L^2}.
	\end{align}
	Applying estimates (\ref{pieceone},\ref{piecetwo},\ref{piecethree}) to the functions $e_k$, from (\ref{increments})
	and Cauchy-Schwarz inequality we get 
	\begin{equation*}
	\expt{\abs{\brak{\omega_{\tau+r}-\omega_\tau, e_k}}^2}\leq C_{\theta,\lambda,T} r |k|^4,
	\end{equation*}
	so that (\ref{fullnorm}) gives us
	\begin{equation*}
	\expt{\norm{\omega_{\tau+r}-\omega_\tau}_{H^{-3-\delta}}^2}\leq \sum_{k\in\Z^2} (1+|k|^2)^{-3-\delta}
	C r \pa{T+|k|^4 T +\lambda}\leq C_{\theta,\lambda,T,\delta} r,
	\end{equation*}
	which concludes the proof.
\end{proof}

\begin{prop}\label{prop:tightness}
	The laws of the processes $\omega_{M,N}:[0,T]\rightarrow H^{-1-\delta}$ are tight in
	\begin{equation*}
	L^q([0,T],H^{-1-\delta})\cap \D([0,T],H^{-3-\delta})
	\end{equation*}
	for any $\delta>0, 1\leq q<\infty$.
\end{prop}

\begin{proof}
	Since $\omega_{M,N,t}\sim \frac{1}{\sqrt N} \Xi_{M+t}^{\theta,N\lambda}$, they are bounded in 
	$L^p(\Omega,H^{-1-\delta})$ for any $\delta>0, 1\leq p<\infty$ uniformly in $M,N,t$ as shown in \autoref{prop:fixedtimemarginals},
	and as a consequence the processes $\omega_{M,N}$ are uniformly bounded in $L^p(\Omega\times [0,T],H^{-1-\delta})$,
	for any $\delta>0, 1\leq p<\infty$. Moreover, we have proved fixed-time tightness. 
	We are thus left to prove Aldous' condition in $H^{-3-\delta}$ and to control a fractional Sobolev norm in time
	in order to apply \autoref{cor:simon} and \autoref{thm:aldous}, concluding the proof.
	As in the previous proof, we denote $\omega_{M,N,t}=\omega_t$.
	
	We only need to apply the uniform bound on increments (\ref{incrementbound}).
	Starting from the fractional Sobolev norm, we evaluate
	\begin{align*}
	\expt{\norm{\omega}_{W^{\alpha,1}([0,T],H^{-3-\delta})}}
	&= \expt{\int_0^T \int_0^T \frac{\norm{\omega_t-\omega_s}_{H^{-3-\delta}}}{|t-s|^{1+\alpha}}dtds}\\
	&\leq \int_0^T \int_0^T \frac{\expt{\norm{\omega_t-\omega_s}_{H^{-3-\delta}}}}{|t-s|^{1+\alpha}}dtds\\
	&\leq C \int_0^T \int_0^T |t-s|^{-1/2-\alpha},
	\end{align*}
	which converges as soon as $\alpha<1/2$. Aldous's condition follows from \v{C}eby\v{s}\"ev's inequality:
	if $\tau$ is a stopping time for $\omega_t$, then
	\begin{equation*}
	\sup_{0\leq r\leq R} \PP\pa{\norm{\omega_{\tau+r}-\omega_\tau}_{H^{-3-\delta}}\geq \epsilon}
	\leq  \epsilon^{-1} \sup_{0\leq r\leq R}  \expt{\norm{\omega_{\tau+r}-\omega_\tau}_{H^{-3-\delta}} }\\
	\leq C \epsilon^{-1} R^{1/2},
	\end{equation*}
	where the right-hand side is smaller than $\epsilon'>0$ as soon as $R$, which we can choose, is small enough.	
\end{proof}

Let us conclude this paragraph with a martingale central limit theorem concerning the driving noise
of our approximant processes.

\begin{prop}\label{prop:martclt}
	Let $(\Pi^N_t)_{t\in[0,T],N\in\N}$ be a sequence of $H^{-1-\delta}$-valued
	martingale with laws $\Pi^N \sim \frac{1}{\sqrt N}\Sigma^{N\lambda}$ (fix $\delta>0$).
	The laws of $\Pi^N$ are tight in
	\begin{equation}\label{thespace}
	L^q([0,T],H^{-1-\delta})\cap \D([0,T],H^{-1-\delta})
	\end{equation}
	for any $\delta>0, 1\leq q<\infty$, and limit points have the law of the Wiener process $\sqrt{\lambda}W_t$ on 
	$H^{-1-\delta}$ with covariance
	\begin{equation*}
	\expt{\brak{W_t,f},\brak{W_s,g}}=t\wedge s \brak{f,g}_{L^2(\T^2)}.
	\end{equation*}
\end{prop}

\begin{proof}
	By (\ref{piecethree}) we readily get
	\begin{equation*}
	\expt{\norm{\Pi^N_{\tau+r}-\Pi^N_\tau}_{H^{-1-\delta}}^2}\leq C_{\delta,\lambda} r
	\end{equation*}
	for any $N\in\N, \delta,r>0$ and any $\tau$ stopping time for $\Pi^N$, uniformly in $N$.
	The very same argument of the last proposition (here with a better space regularity) proves then the claimed tightness.
	The martingale property (with respect to the processes own filtrations) carries on to limit points since it can
	be expressed by means of the following integral formulation: for any $s,t\in[0,T]$,
	\begin{equation*}
	\expt{(\Pi^N_t-\Pi^N_s)\Phi(\Pi^N\mid_{[0,s]})}=0
	\end{equation*}
	for all the real bounded measurable functions $\Phi$ on $(H^{-1-\delta})^{[0,s]}$.
	Limit points are Gaussian processes, since at any fixed time 
	\begin{equation*}
	\frac{1}{\sqrt N} \Sigma^{N\lambda}_t\sim 	\frac{1}{\sqrt N}\Xi_t^{\theta=0,N\lambda}
	\xrightarrow{N\rightarrow\infty} \sqrt{\lambda t} \eta\sim \sqrt{\lambda} W_t,
	\end{equation*}
	as one can show by repeating the computations on characteristic functions in \autoref{prop:fixedtimemarginals} with
	$\theta=0, M=t$. It now suffices to recall the covariance formulas (\ref{whitenoisecovariance}) and (\ref{poissoncovariance}),
	\begin{equation*}
	\expt{\brak{\frac{1}{\sqrt N}\Sigma_t^{N\lambda},f}\brak{\frac{1}{\sqrt N}\Sigma_s^{N\lambda},g}}
	=\lambda (t\wedge s) \brak{f,g}_{L^2}^2
	=\expt{\brak{\sqrt \lambda W_t,f},\brak{\sqrt \lambda W_s,g}},
	\end{equation*}
	to conclude that any limit point has the law of $\sqrt{\lambda} W$.
\end{proof}
\subsection{Identifying Limits}\label{ssec:limitpoints}

The last step is to prove that limit points of the family of processes $\omega_{M,N}$ satisfy \autoref{defi:modelsolution}.
First, let us recall once again our setup for the sake of clarity:
\begin{itemize}
	\item $\lambda,\theta>0$ are fixed throughout;
	\item there is a probability space $(\Omega,\F,\PP)$ on which the stochastic processes
	$\Sigma_t^{N\lambda}$ and the random variables $\Xi_{M}^{\theta,N\lambda}$ are defined, for $M\geq 0, N\in\N$,
	their laws being as in \autoref{sec:definitions};
	\item the processes $(\omega_{M,N,t})_{t\in[0,T]}$ are defined as at the beginning of this section:
	strong solutions of (\ref{eqomegamn}) with initial datum $\frac{1}{\sqrt N}\Xi_{M}^{\theta,N\lambda}$ and driving noise 
	$\frac{1}{\sqrt N}\Sigma_t^{N\lambda}$, built as in \autoref{cor:finitevortices}.
\end{itemize}
To fix notation, let us consider separately the following three cases: by \autoref{prop:tightness}, 
we can consider converging sequences
\begin{itemize}
	\item[(Ps)] $(\omega_{M_n,N=1})_{n\in\N}$, with $M_n\rightarrow\infty$ as $n\rightarrow\infty$,
	the limit being $\omega^P_t$;
	\item[(G)] $(\omega_{M,N_n})_{n\in\N}$, with  $N_n\rightarrow\infty$ as $n\rightarrow\infty$ and fixed $M<\infty$,
	the limit being $\omega^G_{M,t}$;
	\item[(Gs)] $(\omega_{M_n,N_n})_{n\in\N}$, with  $M_n,N_n\rightarrow\infty$ as $n\rightarrow\infty$,
	the limit being $\omega^G_t$;
\end{itemize}
the convergence in law takes place in $L^q([0,T],H^{-1-\delta})\cap \D([0,T],H^{-3-\delta})$,
for any fixed $\delta>0, 1\leq q<\infty$.
By \autoref{prop:fixedtimemarginals}, the Poissonian limit (Ps) has marginals 
$\omega^P_t\sim \Xi_{\infty}^{\theta,\lambda}$, and the Gaussian ones
$\omega^G_{M,t}\sim \sqrt{\frac{\lambda}{2\theta}(1-e^{-2\theta(M+t)})}\eta$ for all $t\in[0,T]$, 
$M\in[0,\infty)$, and $\omega^G_t\sim \sqrt{\frac{\lambda}{2\theta}}\eta$
(the labels are given so to match the ones in \autoref{defi:modelsolution}). 
Notice that $(\omega^P_m)_{m\in\N}$ have all the same driving noise $\Sigma_t^\lambda$,
but different initial data, while in the Gaussian limiting sequences the driving noise also varies.
Let us show that the limit laws in the cases where $M\rightarrow\infty$ are stationary.

\begin{prop}\label{prop:stationarity}
	The processes $\omega^P_t$ and $\omega^G_t$ are stationary.
\end{prop}

\begin{proof}
	As the intuition suggests, the key is the fact that $M$ is a time-like parameter, and taking $M\rightarrow \infty$
	corresponds to the infinite time limit. Formally, we observe that for all $r>0$, $0\leq t_1\leq \dots \leq t_k<\infty$,
	and $M,N$,
	\begin{equation}\label{omogeneity}
	(\omega_{M,N,t_1+r},\dots,\omega_{M,N,t_k+r})\sim (\omega_{M+r,N,t_1},\dots,\omega_{M+r,N,t_k}).
	\end{equation}
	Indeed, by construction (see \autoref{sec:finitevortices}), for all $s<t$, $\omega_{M,N,t}$ is given
	as a measurable function of $\omega_{M,N,s}$ and the driving noise,
	\begin{equation}\label{markovproperty}
	\omega_{M,N,t}=F_{s,t}(\omega_{M,N,s},\Sigma^{N\lambda}\mid_{[s,t]})
	\end{equation}
	this, combined with the fact that $\omega_{M,N,t}\sim \omega_{M+t,N,0}$ and the invariance of $\Sigma^{N\lambda}$
	by time shifts proves (\ref{omogeneity}).
	Passing (\ref{omogeneity}) to the limits (Ps) and (Gs) concludes the proof, since the dependence on $r$ of the 
	right-hand side disappears.
\end{proof}

\begin{rmk}
	Equation (\ref{markovproperty}) is equivalent to the Markov property, \emph{cf.} the beginning of the proof
	to \autoref{cor:finitevortices}. Equation (\ref{omogeneity}) is the time omogeneity property.
	The Markov property is a consequence of uniqueness for the system (\ref{posdyn}), (\ref{intdyn}).
	Since uniqueness result in cases (Ps), (G) and (Gs) of \autoref{defi:modelsolution} seem to be out of reach
	by now, we can not hope to derive the Markov property as well.
\end{rmk}

We are only left to show that our limits do produce the sought solutions of \autoref{thm:mainresult}.
First, we apply Skorokhod's theorem to obtain almost sure convergence. 

\begin{prop}\label{prop:skorokhod}
	There exist stochastic processes $(\tilde \omega^P_n)_{n\in \N},\tilde \Sigma_t^\lambda$,
	defined on a probability space $(\tilde \Omega,\tilde\F,\tilde\PP)$, such that their joint distribution
	coincides with the one of the original objects and with $\tilde \omega^P_m$ converging to 
	a limit $\tilde \omega^P$ almost surely in $L^q([0,T],H^{-1-\delta})\cap \D([0,T],H^{-3-\delta})$
	for any fixed $\delta>0, 1\leq q<\infty$.
	
	Analogously, there exist $(\tilde \omega^G_{M,n},\tilde \omega^G_n,\tilde \Sigma_t^{N_n\lambda})_{n\in \N}$,
	defined on $(\tilde \Omega,\tilde\F,\tilde\PP)$, such that their joint distribution
	coincides with the one of the original objects and with $\tilde \omega^G_{M,n},\tilde \omega^G_n$ 
	converging respectively to limits $\tilde \omega^G_M, \tilde \omega^G$ almost surely in 
	$L^q([0,T],H^{-1-\delta})\cap \D([0,T],H^{-3-\delta})$ for any fixed $\delta>0, 1\leq q<\infty$.
\end{prop}

The proof is a straightforward application of the following version of Skorokhod's theorem, which we borrow from
\cite{motyl} (see references therein).
The required tightness is provided by \autoref{prop:tightness} and \autoref{prop:martclt}.

\begin{thm}[Skorokhod Representation]\label{thm:skorokhod}
	Let $X_1\times X_2$ be the product of two Polish spaces, $\chi^n=(\chi_n^1,\chi_n^2)$ be a sequence of $X_1\times X_2$-valued
	random variables, defined on a probability space $(\Omega,\F,\PP)$, converging in law and such that $\chi_n^1$ have all the same
	law $\rho$. Then there exist a sequence $\tilde \chi^n=(\tilde \chi_n^1,\tilde \chi_n^2)$ of $X_1\times X_2$-valued
	random variables, defined on a probability space $(\tilde \Omega,\tilde\F,\tilde\PP)$, such that
	\begin{itemize}
		\item $\chi^n$ and $\tilde \chi^n$ have the same law for all $n$;
		\item $\tilde \chi^n$ converge almost surely to a $X_1\times X_2$-valued random variable
		$\tilde \chi=(\tilde\chi^n_1,\tilde \chi^n_2)$ on $(\tilde \Omega,\tilde\F,\tilde\PP)$;
		\item the variable $\tilde \chi^n_1$ and $\tilde \chi_1$ coincide almost surely.
	\end{itemize}
\end{thm}
\begin{proof}[Proof of \autoref{prop:skorokhod}.]
	In the case (P) we apply the above result with $X_1=X_2=X=L^q([0,T],H^{-1-\delta})\cap \D([0,T],H^{-3-\delta})$
	and $\chi^m_1=\Sigma_t^\lambda,\chi^m_2=\omega^P_m$, while for the case (G) we take $X_1=\set{0}$ and $X_2=X\times X$,
	with $\chi^n_2=(\omega^G_n,\Sigma_t^{N_n\lambda})$.
\end{proof}

The new processes still are weak solutions of (\ref{eqomegamn}) in the sense of \autoref{defi:modelsolution}.
Consider for instance the $\tilde\omega^G_n$ (the other case being identical): clearly
their trajectories have the same regularity as $\omega^G_n$, and they have the same fixed time distributions.
As for the equation, it holds, for any $f\in C^\infty(\T^2)$ and $t\in[0,T]$, $\PP$-almost surely,
\begin{equation*}
\brak{f,\tilde\omega^G_{n,t}}-\brak{f,\tilde\omega^G_{n,0}}+\theta \int_0^t \brak{f,\tilde\omega^G_{n,s}} ds
-\int_0^t :\brak{H_f,\tilde\omega^G_{n,s}\otimes\tilde\omega^G_{n,s}}: ds -\brak{f,\frac{1}{\sqrt N_n}\Sigma^{N_n\lambda}_t}=0,
\end{equation*}
since taking the expectation of the absolute value (capped by 1) of the right-hand side gives a functional
of the law of $\tilde \omega^G_n, \tilde \Sigma_t^{N_n\lambda}$, which is the same of the original ones.
Moreover, since all the terms in the last equation are \emph{cadlag}
functions in time (in fact they are all continuous but the noise term), one can choose the $\tilde \PP$-full set
on which the equation holds uniformly in $t\in [0,T]$. 

\begin{rmk}
	In fact, one can prove more. Following the proof of Lemma 28 in \cite{flandoli}, it is possible to show that the
	new Skorokhod process have in fact the same point vortices structure of $\omega_{M,N}$, namely it is possible
	to represent $\tilde \omega^P_{m,t}$ and $\tilde \omega^G_{M,n,t},\tilde \omega^G_{n,t}$ as sums of vortices satisfying equations
	(\ref{posdyn}) and (\ref{intdyn}) of \autoref{sec:finitevortices}. The argument would be quite long,
	and we feel that it would not add much to our discussion, so we refrain to go into details,
	contenting us with our analytically weak notion of solution.
\end{rmk}

To ease notation, from now on we will drop all tilde symbols, implying that we are going to work only with the new processes
and noise terms.
We are finally ready to pass to the limit the stochastic equations satisfied by our approximating processes,
thus concluding the proof of our main result.

\begin{proof}[Proof of \autoref{thm:mainresult}.]
	The limits of $\omega^P_n$, $\omega^G_{M,n}$ and $\omega^G_n$ provide respectively the sought solutions in the cases (Ps), (G)
	and (Gs) of \autoref{defi:modelsolution}.
	We focus again our attention on $\omega^G_n$, case (Gs), the other ones being analogous.
	
	Since $\omega^G_n$ converges almost surely in the spaces (\ref{thespace}), we immediately deduce that,
	for any $f\in C^\infty(\T^2)$ and $t\in[0,T]$, $\PP$-almost surely,
	\begin{align}
	\brak{f,\omega^G_{n,t}}&\rightarrow\brak{f,\omega^G_{t}},\\
	\int_0^t \brak{f,\omega^G_{n,s}} ds &\rightarrow \int_0^t \brak{f,\omega^G_{s}} ds. 	
	\end{align}
	The nonlinear term is only slightly more difficult. Let $H_k\in  C^\infty(\T^{2\times 2})$, $k\in\N$, be symmetric
	functions vanishing on the diagonal converging to $H_f$ as $k\rightarrow\infty$
	(it is yet another equivalent of the approximation procedure (\ref{doubleintegralsimplefunctions})).
	Then
	\begin{equation*}
	:\brak{H_k,\omega^G_{n,t}\otimes \omega^G_{n,t}}:=\brak{H_k,\omega^G_{n,t}\otimes \omega^G_{n,t}}
	\rightarrow \brak{H_k,\omega^G_{t}\otimes \omega^G_{t}}=:\brak{H_k,\omega^G_{t}\otimes \omega^G_{t}}:
	\end{equation*} 
	in $L^2(\Omega\times[0,T])$ (the last passage is due to (\ref{doubleintegraldiagonalcontribution})).
	Almost sure convergence of the noise terms is ensured by \autoref{prop:skorokhod}, and the limiting law
	has been determined in \autoref{prop:martclt}, hence, summing up, it holds $\PP$-almost surely
	\begin{equation*}
	\brak{f,\omega^G_{t}}-\brak{f,\omega^G_{0}}+\theta \int_0^t \brak{f,\omega^G_{s}} ds
	-\int_0^t :\brak{H_f,\omega^G_{s}\otimes\omega^G_{s}}: ds -\brak{f,\sqrt \lambda W_t}=0.
	\end{equation*}
	As already noted above, quantifiers in $\PP$ and $t\in[0,T]$ can be exchanged thanks to the fact that
	we are dealing with \emph{cadlag} processes in time.
	Stationarity of $\omega^P_t$ and $\omega^G_t$ follows from \autoref{prop:stationarity}. This concludes the proof of \autoref{thm:mainresult}.
\end{proof}

\renewcommand{\abstractname}{Acknowledgements}
\begin{abstract}
	This article was completed while the author was a Ph.D. student at Scuola Normale Superiore,
	under the supervision of Franco Flandoli, whom the former wishes to thank for many of the ideas
	here exposed.	 
\end{abstract}


\bibliography{grotto}{}
\bibliographystyle{plain}

\end{document}